\documentclass[reqno,twocolumn]{amsart}

\usepackage{amsfonts}
\usepackage{url}
\usepackage{tikz}
\usetikzlibrary{shapes,arrows}
\usetikzlibrary{calc}
\usetikzlibrary{positioning}
\usepackage{minibox}
\usepackage{graphicx}
\usepackage{bm}
\usepackage{flushend}
\usepackage{yhmath}
\usepackage{fullpage}

\tikzstyle{block} = [draw, fill=blue!20, rectangle,
    minimum height=3em, minimum width=4em]
\tikzstyle{sum} = [draw, fill=blue!20, circle, node distance=1cm]
\tikzstyle{input} = [coordinate]
\tikzstyle{output} = [coordinate]

\newcommand\tint{{\textstyle\int\!}}

\newcommand{\liediff}{\mathbin{\triangle}}

\newtheorem{assertion}{Assertion}
\newtheorem{assumption}{Assumption}
\newtheorem{lemma}{Lemma}

\renewcommand{\mathsf}{}

\author{
Stephen Montgomery-Smith}
\address{Department of Mathematics\\
University of Missouri,
Columbia, MO 65211\\
U.S.A.
}

\title{Control of systems by parallel actuators}

\begin{document}


    \begin{abstract}
We describe a particular control method for a system controlled by several actuators with the same control constants.  We show under certain assumptions that the control constants for the whole system can be obtained immediately from the control constants for a single actuator.  This greatly simplifies the work in finding the control constants.  Also, no gain scheduling is required.  The authors have been unable to find any prior work in this direction, and so believe this is a rather new approach.
    \end{abstract}


    \maketitle

\section{Introduction}

In this paper, we describe a system whose state, which we will call the \emph{end effector position}, is given by an element of a differential manifold $\eta \in \mathcal M$.  We suppose that the end effector position is determined by the \emph{values} of $n$ actuators:
\begin{equation}
\bm \ell = (\ell_1, \ell_2,\dots,\ell_n) \in \mathbb R^n .
\end{equation}
We write $\mathsf L:\mathcal M \to \mathbb R^n$ for the function which maps the end effector position to the actuator values.  Note that we can allow the problem to be over-constrained, that is, $n$ can be bigger or equal to the dimension of $\mathcal M$.  Note that the elements of $\mathcal M$ are generalized positions in the sense of the Euler-Lagrange formalism \cite{arnold}.

Examples of these are cable-driven parallel robots, consisting of a fixed rigid frame, and a floating rigid body called the end effector.  The end effector is manipulated via eight cables attached to actuators.  Each actuator is clamped to the fixed frame.  The value, $\ell_k$, of the $k$th actuator is the length of cable issued by the actuator.  The manifold $\mathcal M$ is the six dimensional space of poses, that is, positions with orientations.  See \cite{gallardo-alvarado,gogu,taghirad} for an introduction to parallel robots.  See \cite{pott,qian-et-al} for information about cable-driven parallel robots.  The algorithm described in this paper was tested upon a cable-driven parallel robot built by the NASA Johnson Space Center, which we describe in another paper \cite{montgomery-smith}.

We assume that the only measurements we can take are the values of the actuators, and that the only method of control is to command a force at each actuator, which converts into a force on the end effector.  We assume that once the end effector force is known, one can determine the trajectory of the end effector via a frictionless, quadratic Hamiltonian system.  The force is a generalized force obeying the principle of virtual work, and we will call it the \emph{end effector force}.

The method of control is to calculate the end effector position using the actuator values.  Then, from the difference of the actual end effector position from the requested end effector position, is calculated the required acceleration of the end effector.  From this is calculated the required end effector force.  Finally, we find the actuator forces to effect this.

Our main contribution is to show that the control constants for computing the required acceleration of the end effector are the same as the control constants used to control a single actuator, making it easier to determine the control method.

The benefit of this approach was demonstrated when developing the cable-driven parallel robot that is to be described in \cite{montgomery-smith}.  The control constants were calculated having access to and performing trials upon only one actuator.  When the controller was tried on the full parallel robot, it worked the first time.  If we had been forced to find the control constants by trial and error for the whole system, this would have been very difficult.

This result requires two things.  First that the actuator response is linear.  This would be violated if, for example, if there is significant non-linear friction.  Second, it requires that the each actuator by itself can be controlled by the same linear controller.  For example, with the cable-driven parallel robot, if there is significant flexing or stretching of cables, this would require a different control method for each actuator because the orientations and/or lengths of the cables aren't necessarily identical.

Since the control constants can be found by considering a controller for a single actuator, then the same control constants can be used irrespective of the end effector position, that is, no gain scheduling is required.  Thus the formula for resonant frequencies can also be calculated from the resonant frequencies of a single actuator.  

We believe the methods and results in Sections~\ref{sec parallel} and~\ref{sec theoretical}, where we show that the control constants come from those described in Section~\ref{control actuator}, are new.  For this reason, we know of no prior work in this direction, and thus we have no references.  But it is possible that these methods were developed in a completely different context, and that the author is simply unaware of them.  (Early versions of this paper were submitted to other control theory journals, and they replied that this paper was out of their purview, suggesting to the author that no-one else had considered this approach.)

\section{A brief primer on manifolds, tangent bundles, and cotangent bundles}

While there is a well known abstract definition of a differential manifold (see, for example \cite{lee}), for our purposes we will consider a \emph{manifold of dimension $m$} to be the inverse image of zero under an infinitely-differentiable function $F:\mathbb R^p \to \mathbb R^{m-p}$, whose Jacobian matrix is full rank in a neighborhood of $\mathcal M = F^{-1}(0)$.

Three examples of manifolds that are typically used in robotics are the manifold of positions $\mathbb R^3$, the manifold of orientations $SO_3$, and the manifold of poses $SE_3$, which by identifying the set of three by three matrices with $\mathbb R^9$, and the set of symmetric three by three matrices with $\mathbb R^6$, are respectively created by
\begin{align}
&\begin{aligned}
&F_1:\mathbb R^3 \to \mathbb R^0, \\ &F_1(x) = 0, \\ &\mathbb R^3 = F_1^{-1}(0); \\ &\
\end{aligned}\\
&\begin{aligned}
&F_2:\mathbb R^9 \to \mathbb R^7, \\ &F_2(R) = (R^T R - I, \det(R) - 1); \\ &SO_3 = F_2^{-1}(0); \\ &\
\end{aligned} \\
&\begin{aligned}
&F_3:\mathbb R^{12} \to \mathbb R^7, \\ &F_3(R,x) = (F_2(R), F_1(x)), \\ &SE_3 = F_3^{-1}(0).
\end{aligned}
\end{align}

To each point $\eta \in \mathcal M$ is associated its \emph{tangent space} $T_\eta \mathcal M$, that is, the $m$-dimensional subspace of $\mathbb R^p$ which is tangent to the manifold at $\eta$.  Thus $T_\eta \mathcal M$ can be thought of as the set of end effector velocities.  We call the disjoint union of all tangent spaces the \emph{tangent bundle} $T\mathcal M = \bigcup_{\eta \in \mathcal M} T_\eta \mathcal M$.

Also associated to each point $\eta \in \mathcal M$ is the dual space to the tangent space, which is called the \emph{cotangent space}, and is denoted $T^*_\eta\mathcal M$.  Thus $T^*_\eta\mathcal M$ can be thought of as the set of end effector forces that can be applied to the end effector, with the duality defined as mapping an end effector velocity and an end effector to their inner product, the rate of change of work done by the end effector.  The disjoint union of the cotangent spaces is called the \emph{cotangent bundle} $T\mathcal M = \bigcup_{\eta \in \mathcal M} T_\eta \mathcal M$.

In the first example $\mathbb R^3 = F_1^{-1}(0)$, the tangent space is the set of translational velocities, and the cotangent space is the set of forces.  In the second example $SO_3 = F_2^{-1}(0)$, the tangent space can be identified with the set of angular velocities, and the cotangent space with the set of torques.  In the third example $SE_3 = F_3^{-1}(0)$, the tangent space can be identified with the set of twists, and the cotangent space with the set of wrenches.

\section{Mathematical description of the system}

Given the position of the end effector that is a function of time, $\eta(t)$, we have its velocity and acceleration which are functions of time taking their values in $T\mathcal M$ given by
\begin{equation}
\varphi = \dot \eta, \quad
\alpha = \dot \varphi .
\end{equation}
We define the linear operator $\mathsf \Lambda_\eta:T_\eta\mathcal M \to \mathbb R^n$ by the directional derivative, which in local coordinates is
\begin{equation}
\mathsf \Lambda_\eta \theta = \theta \cdot \frac {d \mathsf L(\eta)}{d\eta} .
\end{equation}
Often we simply write $\Lambda$ for $\Lambda_\eta$ when there is no confusion.

Then, from the velocity $\varphi$, we can calculate the rate of change of the actuator values:
\begin{equation}
\label{dot l Lambda phi}
\dot{\bm\ell} = \Lambda_\eta \varphi.
\end{equation}

We suppose that there is a linear operator $\mathsf T = \mathsf T_\eta : \mathbb R^n \to T^*\mathcal M$, which converts actuator forces $\bm f = (f_1,f_2,\dots,f_n)$
to the end effector force $\tau$:
\begin{equation}
\tau = \mathsf T_\eta \bm f.
\end{equation}
By the principle of virtual work, we have
\begin{equation}
\varphi \cdot \mathsf T_\eta \bm f = \dot{\bm\ell} \cdot \bm f = \mathsf\Lambda_\eta \varphi \cdot \bm f.
\end{equation}
Hence
\begin{equation}
\mathsf T_\eta = \mathsf\Lambda_\eta^T .
\end{equation}

Next we describe the equations of motion.  Suppose that the system is given by a Lagrangian
\begin{equation}
l(\eta,\varphi) = \tfrac12 \mathsf M_\eta(\varphi,\varphi) - v(\eta),
\end{equation}
where $\mathsf M_\eta = \mathsf M$ is a positive definite bilinear operator on $T_\eta\mathcal M$.  This includes kinetic energy of the actuators
\begin{equation}
\label{kin no load}
\tfrac 12 m_0 |\dot{\bm\ell}|^2,
\end{equation}
where $m_0$ is the effective mass of the each actuator (the notion of `effective' is explained in Section~\ref{control actuator} below).  In local coordinates we can describe $\mathsf\Lambda$ and $\mathsf M$ as matrices, then the kinetic energy of the actuators is given by
\begin{equation}
\tfrac12 m_0 \varphi^T \mathsf \Lambda^T \mathsf \Lambda \varphi,
\end{equation}
and so we must have that the matrix
\begin{equation}
\mathsf M - m_0 \mathsf \Lambda^T \mathsf \Lambda
\end{equation}
is positive semi-definite.

Solving the Euler-Lagrange equations \cite{arnold}, we obtain the equations of motion
\begin{equation}
\tau = \mathsf M \alpha + \mu(\eta) ,
\end{equation}
where in local coordinates
\begin{equation}
\mu(\eta) = \varphi \cdot \frac {\partial \mathsf M_\eta}{\partial \eta} (\varphi,\cdot) - \frac{\partial}{\partial \eta} v(\eta) .
\end{equation}

We define the \emph{no-load forces} to be the actuator forces if the actuators are not attached to the system, that is, the Lagrangian is given simply by equation~\eqref{kin no load}:
\begin{equation}
\bm f_0 = m_0 \ddot{\bm\ell}.
\end{equation}
Thus differentiating equation~\eqref{dot l Lambda phi}, we obtain (in local coordinates)
\begin{equation}
\label{f_0}
\bm f_0 = m_0 \mathsf\Lambda_\eta \alpha + m_0 \varphi \cdot \frac{d\Lambda_\eta}{d\eta} \varphi .
\end{equation}
For the example of the cable-driven parallel robot, the cable tensions are given by
\begin{equation}
\text{cable tensions} = \bm f_0 - \bm f.
\end{equation}

Next, we need inverse functions to $\mathsf L$ and $\mathsf T$, which we call $\mathsf Y$ and $\mathsf F$.  We define the \emph{set of admissible actuator values}, $\mathcal L \subset \mathbb R^n$, to be the range of the function $\mathsf L$.  We suppose that we have a \emph{forward kinematics} function, $\mathsf Y : \mathcal L \to \mathcal M$, which is a left inverse to $\mathsf L$.  Because of possible measurement errors, $\mathsf Y$ should produce decent answers even if the actuator values are merely close to $\mathcal L$.  For example, this could be implemented using the Newton-Raphson Method.

For the inverse function of $\mathsf T$, we need some more definitions.  Given $ \bm f_b, \bm f_0 \in \mathbb R^n$, we suppose that we have a predefined set $\mathcal C_{\bm f_b, \bm f_0} \subset \mathbb R^n$.  Here $\bm f_b$ is the command force required to overcome actuator resistance such as back-EMF, $\bm f_0$ is the no-load actuator forces, and $\mathcal C_{\bm f_b, \bm f_0}$ is the set of those $\bm f$ such that it is permissible to command forces $\bm f + \bm f_b$ to the actuators.

Typically this is a convex set defined by a finite number of linear constraints.  For the example of cable-driven parallel robots, we might say $\bm f \in \mathcal C_{\bm f_b, \bm f_0}$ if and only if the tensions in the cables, $\bm f_0 - \bm f$, is never below a given predefined value, and the command forces $\pm(\bm f + \bm f_b)$ don't exceed the actuator hardware limits.

Then we define the \emph{wrench set} to be the set of achievable forces:
\begin{multline}
\mathcal W = \{ (\tau, \bm f_b, \bm f_0) \in T^*\mathcal M \times \mathbb R^n \times \mathbb R^n: \\ \exists \bm f \in \mathcal C_{\bm f_b, \bm f_0} \text{ such that } \mathsf T \bm f = \tau\} .
\end{multline}
We suppose that we have a function $\mathsf F: \mathcal W \times \mathbb R^n \times \mathbb R^n \to \mathbb R^n$ that provides a right inverse to the map defined by $\mathsf T$ in the following manner:
\begin{equation}
\mathsf T(\mathsf F(\tau, \bm f_b, \bm f_0)) = \tau,
\end{equation}
such that
\begin{equation}
\mathsf F(\tau, \bm f_b, \bm f_0) \in \mathcal C_{\bm f_b, \bm f_0}.
\end{equation}
Because there are more actuators than the number of degrees of freedom (that is, the dimension of $\mathcal M$, computing the actuator forces in the last step is an over-constrained problem.  For the example of cable-driven cable robots, there are many approaches in the literature \cite{gould-toint,gouttefarde-et-al,pott}.

Finally we need a way to approximate the difference between two end effector positions by an element of the tangent space.  That is, there is a function $\liediff: \mathcal M \times \mathcal M \to T \mathcal M$ such that $\liediff(\eta_1,\eta_2)$ is in $T_{\eta_1}\mathcal M$, so that with respect to a `reasonable' coordinate system about $\eta_1$ that
\begin{equation}
\label{lie diff approximation}
\eta_2 \approx \eta_1 + \liediff(\eta_1,\eta_2) .
\end{equation}
For example, if $\mathcal M$ is a Riemannian manifold, we could define it as the direction of a geodesic from $\eta_1$ to $\eta_2$.  If $\mathcal M$ is a Lie group, we could define it as the direction of a one parameter subgroup from $\eta_1$ to $\eta_2$.  With the example of $SO_3$, the latter can be identified with angle of rotation from one to the other, multiplied by the unit vector in the direction of the axis of this rotation.

\section{Control of a single actuator}
\label{control actuator}

Let $\ell$ denote actuator value, let $f_c$ be the command force given to a single actuator, and $f$ be the actual force supplied by this actuator.

Each actuator has an \emph{effective} no-load mass $m_0$, which is the ratio $f/\ddot \ell$ when there is no load placed upon the actuator.

For the remainder of this section, we suppose that the actuator is carrying a passive load.  We denote by $m$ to be the effective mass of the actuator with this load.  Thus no load corresponds to $m=m_0$, and we always have $m \ge m_0$.  If the actuator is clamped so that it cannot move, this corresponds to $m = \infty$, which is a mathematical idealization representing when the passive load is very large.

For the purpose of making the analogue of these equations and the system controller equations clearer, we shall replace the actuator actual and command forces by actual and command accelerations
\begin{gather}
a = \frac fm = \ddot \ell, \\
a_c = \frac {f_c}m .
\end{gather}

We look for a controller such that, given an actuator value $\ell_r$, attempts to create a command acceleration $a_c$ such that the actual actuator value, $\ell$, is close to $\ell_r$.

First we describe an \emph{open-loop} controller:
\begin{equation}
\label{pre approx open-loop}
f_c = m \ddot \ell_r + k_0 \dot \ell_r,
\end{equation}
or
\begin{equation}
\label{approx open-loop}
a_c = \ddot \ell_r + \frac{k_0}m \dot \ell_r.
\end{equation}
We call $k_0$ the back-EMF constant, since for electric motors this is a likely source of this term.  This controller fails badly if there is any drift or noise in the system, because it makes no attempt to correct for error.

Next, suppose we also have a good \emph{homogeneous closed-loop} controller.  Denote the vector containing all time derivatives of order less than $l$ of $\ell$ by
\begin{equation}\bm \ell = [\ell, \dot \ell, \ddot \ell, \dots, \ell^{(l-1)}]^T .
\end{equation}
The controller is defined by appropriately sized constant matrices $\mathsf A$, $\mathsf B$, $\mathsf C$, and $\mathsf D$ as:
\begin{align}
\label{hom closed-loop}
\dot {\bm x} &= \mathsf A \bm x + \mathsf B \ell \\
\label{hom closed-loop 2}
a_c &= \mathsf C \bm x + \mathsf D \bm \ell.
\end{align}
By a good homogeneous closed-loop controller, we mean that if this is used to control the passively loaded actuator, then $\ell$ converges to $0$ in a manner that is expeditious enough for our application.

An example is a PID controller
\begin{equation}
\label{pid hom closed-loop}
a_c = - (k_i \tint \ell + k_p \ell + k_d \dot \ell),
\end{equation}
But it could be something more complex, such as a cascaded controller, or a linear quadratic Gaussian controller.

Note that if cable stretching or sagging plays a significant role in the cable-driven parallel robot, then this should be able to control a single actuator with a cable with similar stretching or sagging characteristics attached.

The results of this paper don't depend upon what definition of `good' we use.  Our assertion is that the parallel actuator driven robot controller behaves as well as the single actuator controller.  Thus if the user knows the level of precision or stability required for the whole system, they merely have to check these same parameters for the single actuator.

We combine the open-loop and homogeneous closed-loop controller to obtain a good \emph{closed-loop feed-forward} controller, that is, given a requested actuator value $\ell_r$, and its vector of derivatives
\begin{equation}
\bm \ell_r = [\ell_r, \dot \ell_r, \ddot \ell_r, \dots, \ell_r^{(l-1)}]^T ,
\end{equation}
we find the command acceleration $a_c$ such that $\ell$ converges to $\ell_r$ in an expeditious manner.  This can be created by applying the homogeneous closed-loop controller to
\begin{gather}
\label{y_d}
\ell_d = \ell - \ell_r\\
\label{y_d b}
\bm \ell_d = \bm \ell - \bm \ell_r
\end{gather}
to obtain
\begin{align}
\label{closed-loop feed-forward}
\dot {\bm x} &= \mathsf A \bm x + \mathsf B \ell_d \\
\label{closed-loop feed-forward 2}
a_c &= \ddot \ell_r + \frac{k_0}m \dot \ell_r + \mathsf C \bm x + \mathsf D \bm \ell_d.
\end{align}
For example, with the PID controller it is
\begin{equation}
\label{pid closed-loop feed-forward}
a_c = \ddot \ell_r + \frac{k_0}m \dot \ell_r - (k_i \tint \ell_d + k_p \ell_d + k_d \dot \ell_d).
\end{equation}

\section{The controller for the system by parallel actuators}
\label{sec parallel}

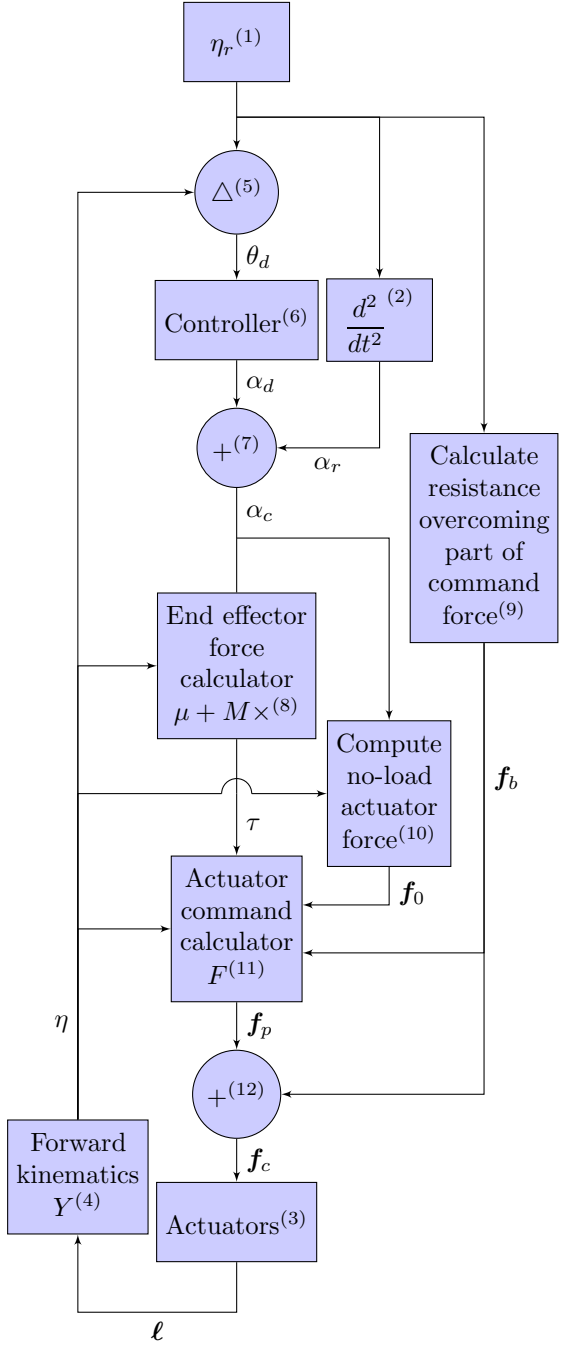
\begin{figure}

\begin{tikzpicture}[auto, node distance=2cm,>=latex']
    \node [block, name=eta_r] {$\eta_r$\textsuperscript{\eqref{first step alg}}};
    \node [input, name=input, below of=eta_r, node distance=1cm] {};
    \node [sum, below of=input, node distance=1cm] (approxliediff) {$\liediff$\textsuperscript{\eqref{liediff}}};
    \draw [->] (eta_r) -- (approxliediff);
    \node [block, below of=approxliediff, node distance=1.7cm] (controller) {Controller\textsuperscript{\eqref{cable control}}};
    \draw [->] (approxliediff) -- node[name=e]{$\theta_d$}(controller);
    \node [sum, below of=controller, node distance=1.7cm] (plus) {$+$\textsuperscript{\eqref{add requested acceleration to control}}};
    \draw [->] (controller) -- node{$\alpha_d$}(plus);
    \coordinate [below of=plus, node distance = 1.2cm] (tilde alpha c) {};
    \draw [-] (plus) -- node {$\alpha_c$} (tilde alpha c);
    \node [block, below of=tilde alpha c, node distance = 1.7cm] (force and torque calculator) {\minibox[c]{End effector\\force\\calculator\\$ \mu + \mathsf M \times$\textsuperscript{\eqref{determine torque-force}}}};
    \coordinate [below of=force and torque calculator, node distance = 1.7cm] (intersection);
    \draw [-] (tilde alpha c) -- (force and torque calculator);
    \node [block, below of=intersection, node distance = 1.8cm] (actuator calculator) {\minibox[c]{Actuator\\command\\calculator\\$\mathsf F$\textsuperscript{\eqref{determine actuator forces}}}};
    \draw [-] (force and torque calculator) -- (intersection);
    \draw [->] (intersection) -- node {$\tau$}(actuator calculator);
    \node [sum, below of=actuator calculator, node distance=2.2cm] (sum forces) {$+$\textsuperscript{\eqref{command cables}}};
    \draw [->] (actuator calculator) -- node{$\bm f_p$} (sum forces);
    \node [block, below of=sum forces, node distance = 1.7cm] (actuators) {Actuators\textsuperscript{\eqref{measure}}};
    \draw [->] (sum forces) -- node{$\bm f_c$} (actuators);

    \node [block, above left=-0.7cm and 0.1cm of actuators] (forward kinematics) {\minibox[c]{Forward\\kinematics\\$\mathsf Y$\textsuperscript{\eqref{forward kinematics}}}};
    \draw [->] (actuators) -- ($(actuators)+(0,-1.2)$)
                           -| node[pos=0.25] {$\bm\ell$} (forward kinematics);

    \node [block, right=0.1cm of controller] (d2/dt2) {\minibox[c]{$\dfrac{d^2}{dt^2}^{\eqref{second derivative}}$}};
    \draw [->] (input) -| (d2/dt2);
    \draw [->] (d2/dt2) |- node[pos=0.75]{$\alpha_r$} (plus);

    \node [block, right=1.2cm of intersection] (no-load force){\minibox[c]{Compute\\no-load\\actuator\\force\textsuperscript{\eqref{calculate no-load actuator forces}}}};
	\draw [->] (tilde alpha c) -| (no-load force);
    \draw [->] (no-load force) |- node[pos=0.35]{$\bm f_0$} (actuator calculator.20);

    \draw [->] (forward kinematics) |- (approxliediff);
    \draw [->] (forward kinematics) |- (force and torque calculator);
    \draw [->] (forward kinematics) |- node [pos=0.25]{$\eta$} (actuator calculator);
    \draw [-] (forward kinematics) |- ($(intersection)-(0.2,0)$);
    \draw [->] ($(intersection)+(0.2,0)$) -- (no-load force);
    \draw ($(intersection)-(0.2,0)$) arc (180:0:0.2);

    \node [block, right=2.3cm of tilde alpha c] (rate of change of actuator value) {\minibox[c]{Calculate\\resistance\\overcoming\\part of\\command\\force\textsuperscript{\eqref{back-emf}}}};
    \draw [->] (input) -| (rate of change of actuator value);
    \draw [->] (rate of change of actuator value) |-  node[pos=0.15] {$\bm f_b$} (sum forces);
    \draw [->] (rate of change of actuator value) |- (actuator calculator.-20);
\end{tikzpicture}

\caption{Block diagram of the controller for the robot driven by parallel actuators.}
\label{controller for cable-driven parallel robot}
\end{figure}

For the controller, we introduce the state vector $\bm \xi$, which is a vector of elements from the tangent space, with the same number of components as the state vector $\bm x$ described in equation~\eqref{hom closed-loop}.  We denote a matrix multiplied by a vector of elements from the tangent space as giving another vector of elements of the tangent space as follows:
\begin{equation}
(\mathsf A \bm \xi)_i := \sum_j \mathsf A_{i,j} \xi_j,
\end{equation}
where $\xi_j$ means the $j$th component of $\bm\xi$.

The control loop is shown as a block diagram in Figure~\ref{controller for cable-driven parallel robot}.  It is labeled throughout with superscript numbers that correspond to the steps given below.
\begin{enumerate}
\item \label{first step alg} Obtain the requested end effector position $\eta_r$.
\item \label{second derivative} Compute the requested acceleration
\begin{equation}
\alpha_r = \ddot\eta_r .
\end{equation}
\item \label{measure} Measure actuator values $\bm\ell$.
\item \label{forward kinematics} Calculate the actual end effector position:
\begin{equation}
\eta = \mathsf Y(\bm\ell).
\end{equation}
\item \label{liediff}
Find the $\liediff$ difference between the actual end effector position and the requested end effector position:
\begin{equation}
\label{theta_d from liediff}
\theta_d = \liediff(\eta, \eta_r)
\end{equation}
and compute the vector of derivatives
\begin{equation}
\bm \theta_d = [\theta_d, \dot \theta_d, \ddot \theta_d, \dots, \theta_d^{(l-1)}]^T .
\end{equation}
\item \label{cable control} Calculate the control part of the command end effector acceleration:
\begin{align}
\label{closed-loop feed-forward robot alpha}
\dot{\bm \xi} &= \mathsf A \bm \xi + \mathsf B \theta_d \\
\label{closed-loop feed-forward robot alpha 2}
\alpha_d &= \mathsf C \bm \xi + \mathsf D \bm \theta_d.
\end{align}
For example, the PID controller would be:
\begin{equation}
\alpha_d = - (k_i \tint \theta_d + k_p \theta_d + k_d \dot \theta_d).
\end{equation}
\item \label{add requested acceleration to control} Calculate the command end effector acceleration:
\begin{equation}
\alpha_c = \alpha_r + \alpha_d.
\end{equation}
\item \label{determine torque-force} Determine the command end effector force to be applied to the end effector:
\begin{equation}
\tau_c = \bm \mu + \mathsf M \alpha_c.
\end{equation}
\item \label{back-emf} Calculate the requested rate of change of the lengths of the cables
\begin{equation}
\dot{\bm \ell}_r = \mathsf\Lambda(\eta_r) \varphi_r,
\end{equation}
and find the resistance overcoming part of the command force for the actuators
\begin{equation}
\label{f_b}
\bm f_b = k_0 \dot{\bm\ell}_r .
\end{equation}
\item \label{calculate no-load actuator forces} Use equation~\eqref{f_0} to calculate the no-load actuator forces.
\begin{equation}
\bm f_0 = m_0 \mathsf\Lambda_\eta \alpha_c + m_0 \varphi \cdot \frac{d\Lambda_\eta}{d\eta} \varphi .
\end{equation}
\item \label{determine actuator forces} Determine whether $(\tau_c,\bm f_b,\bm f_0) \in \mathcal W$.  If it isn't, declare that the end effector is out of its workspace, command the actuators to brake, and quit.  Otherwise, calculate part of the actuator forces using
\begin{equation}
\label{f_p}
\bm f_p = \mathsf F(\tau_c, \bm f_b,\bm f_0) .
\end{equation}
\item \label{command cables} Command the actuators with
\begin{equation}
\label{f_c}
\bm f_c = \bm f_p + \bm f_b .
\end{equation}
\item Go back to Step~\ref{first step alg}.
\end{enumerate}
Note that in Steps~\ref{determine torque-force}, \ref{calculate no-load actuator forces}, and \ref{determine actuator forces}, the quantities $\mathsf M$, $\boldsymbol\mu$, $\mathsf\Lambda$, $\mathcal W$, and $\mathsf F$ are calculated from $\eta$ and $\varphi$, but they could just as well be calculated from $\eta_r$ and $\varphi_r$, the only change required being to equation~\eqref{theta_d from liediff} to
\begin{equation}
\theta_d = - \liediff(\eta_r, \eta)
\end{equation}

\section{Theoretical justification for the controller}
\label{sec theoretical}

\begin{assumption}
\label{assume linear}
The interaction between the command force $f_c$, the actual force $f$, and the actuator value $\ell$ of a single actuator, is given by the linear system
\begin{multline}
\label{system}
f + c_2 \dot f + \cdots + c_n f^{(n-1)} \\
+ k_0 \dot \ell + k_1 \ddot \ell + \cdots + k_n \ell^{(n+1)} \\
= f_c + \tilde c_2 \dot f_c + \cdots + \tilde c_p f_c^{(p-1)}.
\end{multline}
\end{assumption}

Note this linearity can be difficult to achieve if friction is significant in the actuators.  Friction is highly non-linear, especially when $\dot \ell$ and $f$ switch between having the same sign and having different signs, as could happen with an active load.

If the actuator has a passive load $m$ as described in Section~\ref{control actuator}, then equation~\eqref{system} becomes
\begin{multline}
\label{system passive load}
\ddot \ell + c_2 \ell^{(3)} + \cdots + c_n \ell^{(n+1)} \\
+ \tfrac{k_0}m \dot \ell + \tfrac{k_1}m \ddot \ell + \cdots + \tfrac{k_n}m \ell^{(n+1)} \\
= a_c + \tilde c_2 \dot a_c + \cdots + \tilde c_p a_c^{(p-1)}.
\end{multline}
The original open-loop controller was derived from the assumption that equation~\eqref{system} is a perfect description of the passively loaded actuator:
\begin{multline}
f_c + \tilde c_2 \dot f_c + \cdots + \tilde c_p f_c^{(p-1)} \\
= m(\ddot \ell_r + c_2 \ell_r^{(3)} + \cdots + c_n \ell_r^{(n+1)}) \\
+ k_0 \dot \ell_r + k_1 \ddot \ell_r + \cdots + k_n \ell_r^{(n+1)} ,
\end{multline}
which in Section~\ref{control actuator} was approximated with equation~\eqref{pre approx open-loop}.

\begin{assumption}
\label{assume closed-loop feed-forward}
For every $m \in [m_0, \infty]$, equations~\eqref{y_d}, \eqref{y_d b}, \eqref{closed-loop feed-forward} and~\eqref{closed-loop feed-forward 2} provide a good closed-loop feed-forward controller for system~\eqref{system passive load}.
\end{assumption}

This assumption can either be tested theoretically using eigenvalue analysis, or experimentally by loading various passive loads onto a single actuator.  The latter approach doesn't require any knowledge of the coefficients in equation~\eqref{system}.  We simply need to believe that such an equation exists.

\begin{assumption} \label{time scale} The time scale of the corrections $\theta_d$ is much smaller than the time change of $\eta_r$, and the magnitude of $\theta_d$ and its derivatives are much smaller than that of $\eta$ and its corresponding derivatives.  This means we can assume that the time derivatives of $\eta$ are negligible compared to $\eta$, and hence we can assume the matrices $\mathsf M$, $\mathsf F$, and the covector $\mu$, and their derivatives, are constant in the time scales in which the controller operates.  We also assume that equation~\eqref{lie diff approximation} holds for time derivatives of both sides, and that the constants are uniformly controlled in the ranges achieved.
\end{assumption}

The main result of this paper, which we state below, is described as an `assertion' rather than a `theorem,' as the proofs are not very rigorous.

\begin{assertion}
\label{main}
Given Assumptions~\ref{assume linear}, \ref{assume closed-loop feed-forward}, and~\ref{time scale}, the algorithm described in Section~\ref{sec parallel} is a good controller for the parallel actuator driven robot.
\end{assertion}

Pick a time $t_0$ which is in the range of times in which the controller performs the required corrections.  Let $\eta_0 = \eta(t_0)$.  Define
\begin{gather}
\theta = \liediff(\eta, \eta_0) \\
\theta_r = \liediff(\eta_r, \eta_0) .
\end{gather}
Thus
\begin{gather}
\label{approx eta}
\eta \approx \eta_0 + \theta \\
\label{approx eta r}
\eta_r \approx \eta_0 + \theta_r ,
\end{gather}
and if we set
\begin{equation}
\varphi_r = \dot \eta
\end{equation}
then
\begin{gather}
\label{varphi}
\varphi \approx \dot \theta \\
\label{varphi_r}
\varphi_r \approx \dot \theta_r ,
\end{gather}
and we have
\begin{equation}
\label{theta d theta theta r}
\theta_d \approx \theta - \theta_r.
\end{equation}

Let $n_{\mathcal M}$ be the dimension of $\mathcal M$, and define the $(n_{\mathcal M} \times n_{\mathcal M})$ matrix
\begin{equation}
\mathsf N = \mathsf M^{-1} \mathsf\Lambda^T \mathsf\Lambda.
\end{equation}

\begin{assertion}
\label{pre main}
With the same assumptions as Assertion~\ref{main}, if the end effector is controlled by the algorithm given in Section~\ref{sec parallel}, then we have
\begin{multline}
\label{system passive load alpha}
\ddot\theta + c_2 \theta^{(3)} + \cdots + c_n \theta^{(n+1)} \\
+ \mathsf N (k_0 \dot\theta + k_1 \ddot\theta + \cdots + k_n \theta^{(n+1)}) \\
\approx \alpha_c  + \tilde c_2 \dot \alpha_c + \cdots + \tilde c_p \alpha_c^{(p-1)}
\end{multline}
where
\begin{equation}
\alpha_c \approx \hat\alpha_c + k_0 \mathsf N \dot\theta_r
\end{equation}
and $\hat\alpha_c$ is calculated thus:
\begin{equation}
\bm \theta_d = [\theta_d, \dot \theta_d, \ddot \theta_d, \dots, \theta_d^{(l-1)}]^T
\end{equation}
\begin{align}
\dot {\bm \xi} &= \mathsf A \bm \xi + \mathsf B \theta_d \\
\hat\alpha_c &= \alpha_r + \mathsf C \bm \xi + \mathsf D \bm \theta_d.
\end{align}
\end{assertion}

\bigskip\noindent {\em Proof:} \ Rewrite equation~\eqref{system} for the $j$th actuator:
\begin{multline}
\label{system for jth}
f_j + c_2 \dot f_j + \cdots + c_n f_j^{(n-1)} \\
+ k_0 \dot \ell_j + k_1 \ddot \ell_j + \cdots + k_n \ell_j^{(n+1)} \\
= f_{c,j} + \tilde c_2 \dot f_{c,j} + \cdots + \tilde c_p f_{c,j}^{(p-1)}.
\end{multline}
From equations~\eqref{dot l Lambda phi}, \eqref{f_b}, \eqref{f_p}, \eqref{f_c}, and~\eqref{varphi_r}, we obtain
\begin{equation}
f_{c,j} = \mathsf F_j(\tau_c, \bar{\bm f}_c) + k_0 \Lambda \dot\theta,
\end{equation}
and applying $\mathsf T = \mathsf \Lambda^T$ we obtain
\begin{equation}
\mathsf T \bm f_c = \tau_c + k_0 \mathsf \Lambda^T \mathsf \Lambda \dot\theta_r .
\end{equation}
Similarly, from equations~\eqref{dot l Lambda phi} and~\eqref{varphi}, we have
\begin{equation}
\mathsf T \bm f = \tau,
\end{equation}
where $\tau$ is the actual end effector force, and
\begin{equation}
\mathsf T \dot{\bm \ell} = \mathsf \Lambda^T \mathsf \Lambda \dot\theta .
\end{equation}
Hence from Assumption~\ref{time scale}, we obtain
\begin{multline}
\tau + c_2 \dot \tau + \cdots + c_n \tau^{(n-1)} \\
+ \mathsf\Lambda^T \mathsf\Lambda (k_0 \dot\theta + k_1 \ddot\theta + \cdots + k_n \theta^{(n+1)}) \\
\approx  \tau_c + k_0 \mathsf\Lambda^T \mathsf\Lambda \dot\theta_r
+ \tilde c_2 (\dot\tau_c + k_0 \mathsf\Lambda^T \mathsf\Lambda \ddot\theta_r)
+ \\
\cdots + \tilde c_p (\tau_c^{(p-1)} + k_0 \mathsf\Lambda^T \mathsf\Lambda \theta_r^{(p)}) .
\end{multline}
Next, subtracting $\mu$, and then left multiplying by $\mathsf M^{-1}$, and using Assumption~\ref{time scale} again, we obtain equation~\eqref{system passive load alpha}.  The rest of the assertion follows by equation~\eqref{theta d theta theta r}.\hfill Q.E.D.

\begin{lemma}
\label{eig}
The matrix $\mathsf N$ has a basis of eigenvectors, with eigenvalues in $[0,m_0^{-1}]$.
\end{lemma}

\bigskip\noindent {\em Proof:} \  Let
\begin{equation}
\tilde{\mathsf N} = \mathsf M^{-1/2} \mathsf \Lambda^T \mathsf \Lambda \mathsf M^{-1/2} .
\end{equation}
Clearly $\tilde{\mathsf N}$ is symmetric and positive semi-definite, and since $\mathsf M - m_0 \mathsf \Lambda^T \mathsf \Lambda$ is positive definite, we have
\begin{equation}
m_0^{-1} \mathsf I - \tilde{\mathsf N} = m_0^{-1} \mathsf M^{-1/2} (\mathsf M - m_0 \mathsf \Lambda^T \Lambda) \mathsf M^{-1/2}
\end{equation}
is positive definite.  (Here $\mathsf I$ denotes the $(n_{\mathcal M}\times n_{\mathcal M})$ identity matrix.)  Hence $\tilde{\mathsf N}$ has a basis of eigenvectors, with eigenvalues in $[0,m_0^{-1}]$.  Also,
\begin{equation}
\mathsf N = \mathsf M^{-1/2} \tilde{\mathsf N} \mathsf M^{1/2} ,
\end{equation}
and hence $\mathsf N$ and $\tilde{\mathsf N}$ are similar matrices.
\hfill Q.E.D.

\smallskip

\bigskip\noindent {\em Proof of Assertion~\ref{main}:} \  We use Lemma~\ref{eig} to obtain 
$\sigma_1$, $\sigma_2,\dots,\sigma_{n_{\mathcal M}}$ a basis of eigenvectors of $\mathsf N$, with corresponding eigenvalues
$m_1^{-1}$, $m_2^{-1}, \dots, m_{n_{\mathcal M}}^{-1}$, where $m_1$, $m_2, \dots, m_{n_{\mathcal M}} \in [m_0,\infty]$.  We also form a dual basis $\pi_1$, $\pi_2,\dots,\pi_{n_{\mathcal M}}$ that satisfies
\begin{equation}
\sigma_i \cdot \pi_j = \begin{cases} 1 & \text{if $i=j$} \\ 0 & \text{if $i \ne j$,}\end{cases}
\end{equation}
so that for any $\psi \in \mathbb R^{n_{\mathcal M}}$ we have
\begin{gather}
\label{mu nu}
\psi = \sum_{i=1}^{n_{\mathcal M}} (\psi \cdot \pi_i) \sigma_i \\
\pi_i \cdot \mathsf N \beta = m_i^{-1} (\pi_i \cdot \beta) .
\end{gather}
By Assumption~\ref{time scale}, we can assume that all these vectors and eigenvalues are constant.  For each $1 \le i \le n_{\mathcal M}$, dot product the equations in Assertion~\ref{pre main} by $\pi_i$.  Then it may be seen that the resulting equations satisfy the hypotheses of Assumption~\ref{assume closed-loop feed-forward}, with $y$ replaced by $\theta \cdot \pi_i$, with $y_r$ replaced by $\theta_r \cdot \pi_i$, and with $m$ replaced by $m_i$.  Thus $\theta \cdot \pi_i$ is well controlled by $\theta_r \cdot \pi_i$.

Then it follows by equation~\eqref{mu nu} that $\theta$ is well controlled by $\theta_r$.  Therefore by equations~\eqref{approx eta} and~\eqref{approx eta r}, we have that $\eta$ is well controlled by $\eta_r$.
\hfill Q.E.D.

\section{Conclusions}

We have shown, under certain assumptions, that the control constants for a system of several parallel actuators can be derived simply from knowledge of the control constants of a single actuator.  The benefit of this is that finding the control constants for a single actuator is much easier than finding the control constants for the whole system.  Also, we can show that no gain scheduling is required, that is, the control constants do not require adjusting depending upon where the end effector lies.

The assumptions are (1) that each actuator has a linear response to commands, (2) that the time and distance scales of the control changes are much smaller than the respective scales of the requested position, and (3) that the control constants work well no matter what passive load is placed upon each single actuator.  Also, (4) there is also the implicit assumption that the same control constants work for each of the actuators.

The first assumption seems to be very important, in particular, our limited experience is that this approach does not work well if the effect of friction is large.  The fourth assumption means that this approach can be flawed if there are other non-proportional links in the mechanism between the actuator and end effector, for example, if there is a stretchable cable between them, and the lengths of the cables are different for different actuators.

Even with these limitations, we do believe that this approach is definitely of theoretical interest.  And perhaps future work can overcome some of these difficulties.


\begin{thebibliography}{99}

\bibitem{arnold} V.I. Arnold, Mathematical Methods of Classical Mechanics, 2nd Ed, Springer-Verlag, 1989.

\bibitem{gallardo-alvarado} Jaime Gallardo-Alvarado, Kinematic Analysis of Parallel Manipulators by Algebraic Screw Theory, Springer, Switzerland, 2016.

\bibitem{gogu} Grigore Gogu, Structural Synthesis of Parallel Robots, Springer, The Netherlands, 2008.

\bibitem{gould-toint} Nicholas I. M. Gould, Philippe L. Toint, (2000). A Quadratic Programming Bibliography (PDF). RAL Numerical Analysis Group Internal Report 2000-1.

\bibitem{gouttefarde-et-al} Marc Gouttefarde, Johann Lamaury, Christopher Reichert, and Tobias Bruckmann, A Versatile Tension Distribution Algorithm for $n$-DOF Parallel Robots Driven by $n + 2$ Cables, IEEE Transactions on Robotics, Vol. 31, no. 6, 2015.

\bibitem{lee} John Lee, Introduction to Smooth Manifolds. Graduate Texts in Mathematics. 218 (Second ed.). New York London: Springer-Verlag. ISBN 978-1-4419-9981-8. OCLC 808682771, 2012.

\bibitem{montgomery-smith} Stephen Montgomery-Smith, Cecil Shy, Joshua Sooknanan, Derek Bankieris, Tension control of Cable-Driven Parallel Robots, in preparation.

\bibitem{pott} Andreas Pott, Cable-Driven Parallel Robots: Theory and Application, Springer, 2018.

\bibitem{qian-et-al} Sen Qian, Bin Zi, Wei-Wei Shang and Qing-Song Xu, A Review on Cable-driven Parallel Robots, Chinese Journal of Mechanical Engineering volume 31, Article number: 66 (2018).

\bibitem{taghirad} Hamid D. Taghirad, Parallel Robots: Mechanics and Control, CRC Press, 2013.

\end{thebibliography}
\end{document}